\documentclass[11pt,a4paper]{article}

\usepackage{inputenc}
\usepackage{amsmath}
\usepackage{bm}
\usepackage{bbold}
\usepackage{amsthm}
\usepackage{enumerate}

\usepackage{hyperref}
\usepackage{url}

\newtheorem{theorem}{Theorem}
\newtheorem{lemma}[theorem]{Lemma}
\newtheorem{corollary}[theorem]{Corollary}

\newtheorem{remark}{Remark}

\title{A Constrained Tropical Optimization Problem: Complete Solution and Application Example}

\author{Nikolai Krivulin\thanks{Faculty of Mathematics and Mechanics, Saint Petersburg State University, 28 Universitetsky Ave., Saint Petersburg, 198504, Russia, 
nkk@math.spbu.ru.}\thanks{The work was supported in part by the Russian Foundation for Humanities under Grant \#13-02-00338.}
}

\date{}

\begin{document}

\maketitle

\begin{abstract}
The paper focuses on a multidimensional optimization problem, which is formulated in terms of tropical mathematics and consists in minimizing a nonlinear objective function subject to linear inequality constraints. To solve the problem, we follow an approach based on the introduction of an additional unknown variable to reduce the problem to solving linear inequalities, where the variable plays the role of a parameter. A necessary and sufficient condition for the inequalities to hold is used to evaluate the parameter, whereas the general solution of the inequalities is taken as a solution of the original problem. Under fairly general assumptions, a complete direct solution to the problem is obtained in a compact vector form. The result is applied to solve a problem in project scheduling when an optimal schedule is given by minimizing the flow time of activities in a project under various activity precedence constraints. As an illustration, a numerical example of optimal scheduling is also presented.
\\

\\

\end{abstract}

\section{Introduction}

Tropical optimization problems form a rapidly evolving research domain in the area of tropical (idempotent) mathematics. Multidimensional optimization problems formulated and solved in the framework of tropical mathematics were apparently first considered in \cite{Cuninghamegreen1976Projections,Superville1978Various} shortly after the pioneering works in the area have made their appearance, including \cite{Pandit1961Anew,Cuninghamegreen1962Describing,Giffler1963Scheduling,Vorobjev1963Theextremal,Romanovskii1964Asymptotic}.

The tropical optimization problems arise in real-world applications in various fields, among them are project scheduling \cite{Cuninghamegreen1976Projections,Cuninghamegreen1979Minimax,Zimmermann1981Linear,Zimmermann1984Some,Zimmermann2006Interval,Butkovic2009Introduction,Butkovic2009Onsome,Aminu2012Nonlinear} and location analysis \cite{Cuninghamegreen1991Minimax,Krivulin2011Analgebraic,Krivulin2011Anextremal,Krivulin2012Anew}. Further examples include solutions to problems in transportation networks \cite{Zimmermann1981Linear,Zimmermann2006Interval}, decision making \cite{Elsner2004Maxalgebra,Gursoy2013Theanalytic} and discrete event systems \cite{Gaubert1995Resource,Deschutter2001Model,Krivulin2005Evaluation}.

The problems are formulated in the tropical mathematics setting to minimize a linear or nonlinear objective function defined on vectors of a finite-dimensional semimodule over an idempotent semifield. Both unconstrained and constrained problems are under consideration, where the constraints have the form of linear vector equations and inequalities in the semimodule.

There are tropical optimization problems that are examined in the literature in terms of particular idempotent semifields, whereas some other problems are solved in a more general context, which includes such semifields as a special case. Related solutions often take the form of iterative numerical procedures that produce a solution, or indicate that no solution exists. In other cases, explicit solutions are obtained in a closed form. Many existing approaches, however, offer particular solutions rather than solve the problems completely.

A direct tropical analog of linear programming problems with a linear objective function and linear inequality constraints is among the long-known and extensively studied optimization problems. Complete direct solutions are obtained for the problem under various algebraic assumptions in \cite{Superville1978Various,Zimmermann1981Linear}.

An extended problem with more constraints is considered in \cite{Zimmermann1984Onmaxseparable,Zimmermann1992Optimization,Zimmermann2003Disjunctive,Zimmermann2006Interval} within the framework of max-separable functions. Explicit solutions for the problem are given basically in conventional terms rather than in terms of tropical vector algebra. 

An optimization problem with a linear objective function and two-sided equality constraints is examined in \cite{Butkovic1984Onproperties,Butkovic2010Maxlinear,Butkovic2009Introduction}, where a pseudo-polynomial algorithm is suggested, which produces a solution if any or indicates that no solution exists. A heuristic approach is developed in \cite{Aminu2012Nonlinear} to get an approximate solution of the problem.

A problem with a nonlinear objective function, which arises in the underestimating approximation in the Chebyshev norm, is examined in \cite{Cuninghamegreen1976Projections}, where a complete explicit solution is given. A similar solution to the problem is suggested by \cite{Zimmermann1981Linear}.

A constrained problem of minimizing a Chebyshev-like distance function is solved by a polynomial-time threshold-type algorithm in \cite{Zimmermann1984Some}. An explicit solution to a problem of minimizing the range norm is given in \cite{Butkovic2009Onsome}. A problem with two-sided equality constraints is solved in \cite{Gaubert2012Tropical} by an iterative computational procedure.

Finally, both unconstrained and constrained problems with nonlinear objective functions formulated in terms of a general idempotent semifield are investigated in \cite{Krivulin2006Eigenvalues,Krivulin2006Solution,Krivulin2009Methods,Krivulin2012Anew,Krivulin2012Atropical,Krivulin2013Amultidimensional}. A solution technique applied in these works is based on new results in tropical spectral theory and solutions of linear tropical equations and inequalities. With this technique, direct explicit solutions are obtained in a compact vector form.

In this paper, we examine a multidimensional optimization problem that extends problems in \cite{Krivulin2005Evaluation,Krivulin2006Eigenvalues,Krivulin2009Methods,Krivulin2011Analgebraic,Krivulin2012Atropical,Krivulin2013Amultidimensional} by eliminating restrictions on matrices involved as well as by introducing additional inequality constraints. The problem originates from project scheduling when an optimal schedule is given by minimizing the flow time of activities in a project under various activity precedence constraints.

We formulate the problem in terms of a general idempotent semifield. We follow the approach proposed in \cite{Krivulin2012Atropical,Krivulin2013Amultidimensional} and based on the introduction of an additional unknown variable to reduce the problem to solving linear inequalities, where the variable plays the role of a parameter. A necessary and sufficient condition for the inequalities to hold is used to evaluate the parameter, whereas the solution of the inequalities is taken as a complete direct solution to the original problem. The solution is given in a vector form suitable for both further analysis and applications. 

The rest of the paper is organized as follows. We start with a short concise introduction to tropical algebra in Section~\ref{S-PDR} to provide a formal basis for subsequent solution of optimization problems. In Section~\ref{S-LI}, we examine a system of simultaneous linear inequalities to solve it in a compact vector form. We formulate a general tropical optimization problem, give complete direct solution to the problem, and consider some special cases in Section~\ref{S-AOP}. Finally, Section~\ref{S-AOS} is concerned with application of the results to solve the motivating scheduling problem. To illustrate, a numerical example of optimal schedule development is also presented.

\section{Preliminary definitions and results}\label{S-PDR}

In this section we present a short overview of basic definitions, notation and preliminary results of tropical (idempotent) mathematics to provide a formal framework for the analysis and solution of  optimization problems in the rest of the paper.

Concise introductions to and comprehensive presentations of tropical mathematics are given in different forms in a range of published works, including recent publications \cite{Kolokoltsov1997Idempotent,Golan2003Semirings,Heidergott2006Maxplus,Litvinov2007Themaslov,Butkovic2010Maxlinear}. In the overview below, we mainly follow \cite{Krivulin2006Solution,Krivulin2006Eigenvalues,Krivulin2009Methods}, which offer the prospect of complete direct solution of the problems of interest in a compact vector form. For additional details on and deep insight into the theory and methods of tropical mathematics one can consult the works listed before.

\subsection{Idempotent semifield}

We consider a commutative idempotent semifield $\langle\mathbb{X},\mathbb{0},\mathbb{1},\oplus,\otimes\rangle$ over a set $\mathbb{X}$, which is closed under addition $\oplus$ and multiplication $\otimes$, and has zero $\mathbb{0}$ and identity $\mathbb{1}$. Both addition and multiplication are associative and commutative operations, and multiplication is distributive over addition.

Addition is idempotent, which implies that $x\oplus x=x$ for any $x\in\mathbb{X}$. The addition induces on $\mathbb{X}$ a partial order such that the relation $x\leq y$ holds for $x,y\in\mathbb{X}$ if and only if $x\oplus y=y$. With respect to the order, the addition is isotone in each argument and has an extremal property that $x\leq x\oplus y$ and $y\leq x\oplus y$.

The partial order is considered as extendable to a total order, and so we assume the semifield to be linearly ordered. In what follows, the relation symbols and minimization problems are thought in terms of this order. Note that, according to the order, $x\geq\mathbb{0}$ for all $x\in\mathbb{X}$.

For each $x\in\mathbb{X}_{+}$, where $\mathbb{X}_{+}=\mathbb{X}\setminus\{\mathbb{0}\}$, there exists an inverse $x^{-1}$ that yields $x^{-1}\otimes x=\mathbb{1}$. For any $x\in\mathbb{X}_{+}$ and integer $p\geq1$, the integer power is routinely defined as $x^{0}=\mathbb{1}$, $x^{p}=x^{p-1}\otimes x$, $x^{-p}=(x^{-1})^{p}$, $\mathbb{0}^{p}=\mathbb{0}$, and $\mathbb{0}^{0}=\mathbb{1}$. We further assume that the power notation can be extended to the rational exponents, and so treat the semiring as radicable. Below, we omit the multiplication symbol for the sake of brevity and employ the power notation only in the sense defined.

As examples of the radicable linearly ordered idempotent semifield one can take $\mathbb{R}_{\max,+}=\langle\mathbb{R}\cup\{-\infty\},-\infty,0,\max,+\rangle$ and $\mathbb{R}_{\min,\times}=\langle\mathbb{R}_{+}\cup\{+\infty\},+\infty,1,\min,\times\rangle$. 

The semifield $\mathbb{R}_{\max,+}$ has addition and multiplication defined, respectively, as maximum and arithmetic addition. It is equipped with the zero $\mathbb{0}=-\infty$ and the identity $\mathbb{1}=0$. Each $x\in\mathbb{R}$ is endowed with an inverse $x^{-1}$ that is equal to $-x$ in the ordinary notation. The power $x^{y}$ is actually defined for any $x,y\in\mathbb{R}$ and coincides with the arithmetic product $xy$. The order induced by idempotent addition corresponds to the natural linear order on $\mathbb{R}$.

In the semifield $\mathbb{R}_{\min,\times}$, we have $\oplus=\min$, $\otimes=\times$, $\mathbb{0}=+\infty$, and $\mathbb{1}=1$. The symbols of taking inverse end exponent have ordinary meaning. Idempotent addition produces a reverse order to the natural order on $\mathbb{R}$.

\subsection{Matrix and vector algebra}

We are now concerned with matrices with entries in $\mathbb{X}$. We denote by $\mathbb{X}^{m\times n}$ the set of matrices having $m$ rows and $n$ columns. A matrix with all entries equal to $\mathbb{0}$ is the zero matrix denoted by $\mathbb{0}$. A matrix is called column-regular if it has no columns consisting entirely of zeros.

Addition and multiplication of conforming matrices, as well as multiplication by scalars are defined in the regular way through the scalar operations on $\mathbb{X}$.

Based on properties of scalar addition and multiplication, the matrix operations are elementwise isotone in each argument. For any matrices $\bm{A}$ and $\bm{B}$ of the same size, the elementwise inequalities $\bm{A}\leq\bm{A}\oplus\bm{B}$ and $\bm{B}\leq\bm{A}\oplus\bm{B}$ are valid as well.

Consider square matrices of order $n$ in the set $\mathbb{X}^{n\times n}$. Any matrix with the off-diagonal entries equal to $\mathbb{0}$ is a diagonal matrix. A diagonal matrix that has all diagonal entries equal to $\mathbb{1}$ presents the identity matrix denoted by $\bm{I}$.

The matrix power with non-negative integer exponents is given in the usual way. For any square matrix $\bm{A}$ and integer $p\geq1$, we have $\bm{A}^{0}=\bm{I}$, $\bm{A}^{p}=\bm{A}^{p-1}\bm{A}$.

The trace of a matrix $\bm{A}=(a_{ij})$ is conventionally defined as
$$
\mathop\mathrm{tr}\bm{A}
=
a_{11}\oplus\cdots\oplus a_{nn}.
$$

It is easy to verify that, for any matrices $\bm{A}$ and $\bm{B}$, and scalar $x$, the trace exhibits the standard properties in the form of equalities
$$
\mathop\mathrm{tr}(\bm{A}\oplus\bm{B})
=
\mathop\mathrm{tr}\bm{A}
\oplus
\mathop\mathrm{tr}\bm{B},
\qquad
\mathop\mathrm{tr}(\bm{A}\bm{B})
=
\mathop\mathrm{tr}(\bm{B}\bm{A}),
\qquad
\mathop\mathrm{tr}(x\bm{A})
=
x\mathop\mathrm{tr}\bm{A}.
$$

As usual, a matrix that has only one column (row) is considered as a column (row) vector. The set of column vectors of order $n$ is denoted by $\mathbb{X}^{n}$. A vector that has all components equal to $\mathbb{0}$ is the zero vector. A vector with nonzero components is called regular. The set of regular vectors in $\mathbb{X}^{n}$ is denoted by $\mathbb{X}_{+}^{n}$. 

For any nonzero vector $\bm{x}=(x_{i})\in\mathbb{X}^{n}$, the multiplicative conjugate transpose is a row vector $\bm{x}^{-}=(x_{i}^{-})$ with entries $x_{i}^{-}=x_{i}^{-1}$ if $x_{i}^{-}>\mathbb{0}$, and $x_{i}^{-}=\mathbb{0}$ otherwise.

Below, we well use some properties of multiplicative conjugate transposition, which are easy to verify. Specifically, for any regular vectors $\bm{x},\bm{y}\in\mathbb{X}^{n}$, the componentwise inequality $\bm{x}\leq\bm{y}$ implies that $\bm{x}^{-}\geq\bm{y}^{-}$ and vice versa.

Moreover, for any nonzero vector $\bm{x}\in\mathbb{X}^{n}$, we have the equality $\bm{x}^{-}\bm{x}=\mathbb{1}$. If the vector $\bm{x}$ is regular, then the inequality $\bm{x}\bm{x}^{-}\geq\bm{I}$ holds as well.

\subsection{Spectral radius}

Every square matrix $\bm{A}\in\mathbb{X}^{n\times n}$ defines a linear operator on $\mathbb{X}^{n}$ with certain spectral properties. As usual, a scalar $\lambda$ is an eigenvalue of $\bm{A}$, if there exists a nonzero vector $\bm{x}$ such that $\bm{A}\bm{x}=\lambda\bm{x}$.

The maximum eigenvalue (in the sense of the order on $\mathbb{X}$) is called the spectral radius of the matrix $\bm{A}$ and given by
$$
\lambda
=
\mathop\mathrm{tr}\nolimits\bm{A}\oplus\cdots\oplus\mathop\mathrm{tr}\nolimits^{1/n}(\bm{A}^{n}).
$$

The spectral radius $\lambda$ of any matrix $\bm{A}\in\mathbb{X}^{n\times n}$ possesses a useful extremal property \cite{Cuninghamegreen1979Minimax,Elsner2004Maxalgebra,Krivulin2005Evaluation}, which says that 
$$
\min\ \bm{x}^{-}\bm{A}\bm{x}
=
\lambda,
$$
where the minimum is over all regular vectors $\bm{x}\in\mathbb{X}^{n}$.

\section{Linear inequalities}\label{S-LI}

Solution to optimization problems in the subsequent sections makes use of complete direct solutions of linear tropical inequalities. This section begins with a presentation of results based on solutions given in \cite{Krivulin2006Solution,Krivulin2009Methods,Krivulin2013Amultidimensional} for linear vector inequalities of two types. Furthermore, a problem of simultaneous solution of a system of linear inequalities is considered, which is of independent interest.

\subsection{Preliminary results}

Given a square matrix $\bm{A}\in\mathbb{X}^{n\times n}$ and a vector $\bm{b}\in\mathbb{X}^{n}$, consider a problem of finding all regular solutions $\bm{x}\in\mathbb{X}^{n}$ of the inequality
\begin{equation}
\bm{A}\bm{x}\oplus\bm{b}
\leq
\bm{x}.
\label{I-Axbx}
\end{equation}

To describe a solution, we introduce a function that takes $\bm{A}$ to a scalar
$$
\mathop\mathrm{Tr}(\bm{A})
=
\mathop\mathrm{tr}\bm{A}\oplus\cdots\oplus\mathop\mathrm{tr}\bm{A}^{n}.
$$

Provided that $\mathop\mathrm{Tr}(\bm{A})\leq\mathbb{1}$, we define a matrix
$$
\bm{A}^{\ast}
=
\bm{I}\oplus\bm{A}\oplus\cdots\oplus\bm{A}^{n-1}.
$$

With this notation, we slightly reformulate a useful result, which is apparently first obtained by \cite{Carre1971Analgebra}. In a new form, the result states that, under the condition $\mathop\mathrm{Tr}(\bm{A})\leq\mathbb{1}$, the inequality
$$
\bm{A}^{k}
\leq
\bm{A}^{\ast}
$$
holds for all integer $k\geq0$, and is referred to below as the Carr{\'e} inequality.

The next assertion provides a general solution to inequality \eqref{I-Axbx}.
\begin{theorem}[\cite{Krivulin2006Solution,Krivulin2013Amultidimensional}]
\label{T-IAxbx}
Let $\bm{x}$ be the general regular solution of inequality \eqref{I-Axbx}. Then the following statements are valid:
\begin{enumerate}
\item If $\mathop\mathrm{Tr}(\bm{A})\leq\mathbb{1}$, then $\bm{x}=\bm{A}^{\ast}\bm{u}$ for all regular vectors $\bm{u}$ such that $\bm{u}\geq\bm{b}$.
\item If $\mathop\mathrm{Tr}(\bm{A})>\mathbb{1}$, then there is no regular solution.
\end{enumerate}
\end{theorem}

We now consider another problem. Given a matrix $\bm{C}\in\mathbb{X}^{m\times n}$ and a vector $\bm{d}\in\mathbb{X}^{m}$, find all regular vectors $\bm{x}\in\mathbb{X}^{n}$ to satisfy the inequality
\begin{equation}
\bm{C}\bm{x}
\leq
\bm{d}.
\label{I-Cxd}
\end{equation}

\begin{lemma}[\cite{Krivulin2009Methods}]
\label{L-ICxd}
A vector $\bm{x}$ is a solution of inequality \eqref{I-Cxd} with a column-regular matrix $\bm{C}$ and regular vector $\bm{d}$ if and only if
$$
\bm{x}
\leq
(\bm{d}^{-}\bm{C})^{-}.
$$
\end{lemma}

\subsection{A system of inequalities}

Consider a problem of simultaneous solution of inequalities \eqref{I-Axbx} and \eqref{I-Cxd} combined into the system
\begin{equation}
\begin{aligned}
\bm{A}\bm{x}\oplus\bm{b}
&\leq
\bm{x},
\\
\bm{C}\bm{x}
&\leq
\bm{d}.
\end{aligned}
\label{I-AxbxCxd}
\end{equation}

A general solution of the system is given by the following statement.
\begin{lemma}\label{L-IAxbxCxd}
Let $\bm{x}$ be the general regular solution of system \eqref{I-AxbxCxd} with a column-regular matrix $\bm{C}$ and regular vector $\bm{d}$. Denote $\Delta=\mathop\mathrm{Tr}(\bm{A})\oplus\bm{d}^{-}\bm{C}\bm{A}^{\ast}\bm{b}$.

Then the following statements hold:
\begin{enumerate}
\item If $\Delta\leq\mathbb{1}$, then $\bm{x}=\bm{A}^{\ast}\bm{u}$, where $\bm{b}\leq\bm{u}\leq(\bm{d}^{-}\bm{C}\bm{A}^{\ast})^{-}$.
\item If $\Delta>\mathbb{1}$, then there is no regular solution.
\end{enumerate}
\end{lemma}
\begin{proof}
It follows from Theorem~\ref{T-IAxbx} that the first inequality has regular solutions if and only if the condition $\mathop\mathrm{Tr}(\bm{A})\leq\mathbb{1}$ holds and that all solutions take a general form $\bm{x}=\bm{A}^{\ast}\bm{u}$ for any regular vector $\bm{u}\geq\bm{b}$.

Assume the above condition is satisfied and take the general solution of the first inequality. Substitution of the solution into the second inequality leads to a new system of inequalities with respect to $\bm{u}$, which is given by
\begin{align*}
\bm{C}\bm{A}^{\ast}\bm{u}
&\leq
\bm{d},
\\
\bm{u}
&\geq
\bm{b}.
\end{align*}

Application of Lemma~\ref{L-ICxd} to the first inequality gives a general solution in the form $\bm{u}\leq(\bm{d}^{-}\bm{C}\bm{A}^{\ast})^{-}$, where the right-hand side is a regular vector, since $\bm{A}^{\ast}\geq\bm{I}$. This solution, combined with the second inequality, results in two-sided boundary conditions in the form $\bm{b}\leq\bm{u}\leq(\bm{d}^{-}\bm{C}\bm{A}^{\ast})^{-}$.

The conditions specify a nonempty set only when $\bm{b}\leq(\bm{d}^{-}\bm{C}\bm{A}^{\ast})^{-}$. It is not difficult to verify that the inequality is equivalent to $\bm{d}^{-}\bm{C}\bm{A}^{\ast}\bm{b}\leq\mathbb{1}$. Indeed, multiplying the first inequality on the left by $\bm{d}^{-}\bm{C}\bm{A}^{\ast}$ directly produces the second. We now take the second inequality, multiply it from the left by $(\bm{d}^{-}\bm{C}\bm{A}^{\ast})^{-}$, and then note that $\bm{b}\leq(\bm{d}^{-}\bm{C}\bm{A}^{\ast})^{-}\bm{d}^{-}\bm{C}\bm{A}^{\ast}\bm{b}\leq(\bm{d}^{-}\bm{C}\bm{A}^{\ast})^{-}$, which yields the first inequality.

Both conditions $\mathop\mathrm{Tr}(\bm{A})\leq\mathbb{1}$ and $\bm{d}^{-}\bm{C}\bm{A}^{\ast}\bm{b}\leq\mathbb{1}$ are combined into one equivalent condition 
$\Delta=\mathop\mathrm{Tr}(\bm{A})\oplus\bm{d}^{-}\bm{C}\bm{A}^{\ast}\bm{b}\leq\mathbb{1}$, which completes the proof.
\end{proof}

\begin{remark}
It is possible to represent $\Delta=\mathop\mathrm{Tr}(\bm{A})\oplus\bm{d}^{-}\bm{C}\bm{A}^{\ast}\bm{b}$, provided that $\Delta\leq\mathbb{1}$, in another form to be exploited below. In fact, in this case it holds that
\begin{equation}
\mathop\mathrm{Tr}(\bm{A})\oplus\bm{d}^{-}\bm{C}\bm{A}^{\ast}\bm{b}
=
\bm{d}^{-}\bm{C}\bm{b}
\oplus
\bigoplus_{m=1}^{n}\mathop\mathrm{tr}(\bm{A}^{m}(\bm{I}\oplus\bm{b}\bm{d}^{-}\bm{C})).
\label{E-TrAdCAb}
\end{equation}

To verify the equality, first note that the condition $\Delta\leq\mathbb{1}$ involves $\mathop\mathrm{Tr}(\bm{A})\leq\mathbb{1}$. It follows from the Carr\'{e} inequality that $\bm{A}^{n}\leq\bm{A}^{\ast}$, and thus $\bm{I}\oplus\bm{A}\oplus\cdots\oplus\bm{A}^{n}=\bm{A}^{\ast}$. The left-hand side is now represented as
$$
\mathop\mathrm{Tr}(\bm{A})\oplus\bm{d}^{-}\bm{C}\bm{A}^{\ast}\bm{b}
=
\bigoplus_{m=1}^{n}\mathop\mathrm{tr}\bm{A}^{m}
\oplus
\bm{d}^{-}\bm{C}\bm{b}
\oplus
\bigoplus_{m=1}^{n}\bm{d}^{-}\bm{C}\bm{A}^{m}\bm{b}.
$$

Inserting the trace operator into the last term and combining both terms involving the trace together lead to the desired result.
\end{remark}

\section{An optimization problem}\label{S-AOP} 

In this section we examine an optimization problem with nonlinear objective function and linear inequality constraints. The problem extends those in \cite{Krivulin2005Evaluation,Krivulin2006Eigenvalues,Krivulin2009Methods,Krivulin2011Analgebraic,Krivulin2012Atropical,Krivulin2013Amultidimensional} by eliminating  restrictions on matrices involved as well as by introducing additional inequality constraints.

\subsection{Problem formulation}

Suppose $\mathbb{X}$ is a linearly ordered radicable idempotent semifield. Given matrices $\bm{A},\bm{B}\in\mathbb{X}^{n\times n}$, $\bm{C}\in\mathbb{X}^{m\times n}$ and vectors $\bm{g}\in\mathbb{X}^{n}$, $\bm{h}\in\mathbb{X}^{m}$, the problem is to find all regular vectors $\bm{x}\in\mathbb{X}^{n}$ that
\begin{equation}
\begin{aligned}
&
\text{minimize}
&&
\bm{x}^{-}\bm{A}\bm{x},
\\
&
\text{subject to}
&&
\bm{B}\bm{x}\oplus\bm{g}
\leq
\bm{x},
\\
&&&
\bm{C}
\bm{x}
\leq
\bm{h}.
\end{aligned}
\label{P-xAxBxgxCxh}
\end{equation}

The problem is actually a further generalization of that examined in \cite{Krivulin2013Amultidimensional}, where only the first inequality constraint from \eqref{P-xAxBxgxCxh} is taken into account.

Consider the inequality constraints. It follows from Lemma~\ref{L-IAxbxCxd} that the constraints may have no common regular solution, and so make the entire problem unsolvable. The lemma gives necessary and sufficient conditions for the inequality to define a nonempty feasible set in the form $\mathop\mathrm{Tr}(\bm{B})\oplus\bm{h}^{-}\bm{C}\bm{B}^{\ast}\bm{g}\leq\mathbb{1}$.

Below, we derive a solution to the problem under fairly general assumptions. Some special cases of the problem are also discussed.

\subsection{The main result}

We now give a complete direct solution to problem \eqref{P-xAxBxgxCxh}, which is based on the approach suggested in \cite{Krivulin2012Atropical,Krivulin2013Amultidimensional}. Here, the approach is further developed to handle the new system of inequality constraints through the use of the solution given above for a system of linear inequalities.

We introduce an auxiliary variable to represent the minimum value of the objective function, and then reduce the problem to solution of a system of inequalities, where the variable has the role of a parameter. Necessary and sufficient conditions for the system to have regular solutions are used to evaluate the parameter. Finally, a general solution to the system is exploited as a general solution of the problem.

\begin{theorem}\label{T-xAxBxgxCxh}
Suppose that $\bm{A}$ is a matrix with spectral radius $\lambda>\mathbb{0}$. Let $\bm{B}$ be a matrix, $\bm{C}$ be a column-regular matrix, $\bm{g}$ be a vector, and $\bm{h}$ be a regular vector such that $\mathop\mathrm{Tr}(\bm{B})\oplus\bm{h}^{-}\bm{C}\bm{B}^{\ast}\bm{g}\leq\mathbb{1}$. Define a scalar
\begin{equation}
\theta
=
\bigoplus_{k=1}^{n}\mathop{\bigoplus\hspace{0.0em}}_{0\leq i_{0}+i_{1}+\cdots+i_{k}\leq n-k}\mathop\mathrm{tr}\nolimits^{1/k}(\bm{B}^{i_{0}}(\bm{A}\bm{B}^{i_{1}}\cdots\bm{A}\bm{B}^{i_{k}})(\bm{I}\oplus\bm{g}\bm{h}^{-}\bm{C})).
\label{E-theta}
\end{equation}

Then the minimum in \eqref{P-xAxBxgxCxh} is equal to $\theta$ and attained if and only if
\begin{equation}
\bm{x}
=
(\theta^{-1}\bm{A}\oplus\bm{B})^{\ast}\bm{u},
\qquad
\bm{g}
\leq
\bm{u}
\leq
(\bm{h}^{-}\bm{C}(\theta^{-1}\bm{A}\oplus\bm{B})^{\ast})^{-}.
\label{E-xthetaABu-guhCthetaAB}
\end{equation}
\end{theorem}

\begin{proof}
Since the inequality $\mathop\mathrm{Tr}(\bm{B})\oplus\bm{h}^{-}\bm{C}\bm{B}^{\ast}\bm{g}\leq\mathbb{1}$ is valid by the conditions of the theorem, the feasible set of regular vectors in the problem is not empty. Note that the condition implies both inequalities $\mathop\mathrm{Tr}(\bm{B})\leq\mathbb{1}$ and $\bm{h}^{-}\bm{C}\bm{B}^{\ast}\bm{g}\leq\mathbb{1}$.

Denote by $\theta$ the minimum of the objective function on the feasible set and note that $\theta\geq\lambda>\mathbb{0}$. Any regular $\bm{x}$ that yields the minimum must satisfy the system
\begin{align*}
\bm{x}^{-}\bm{A}\bm{x}
&=
\theta,
\\
\bm{B}\bm{x}\oplus\bm{g}
&\leq
\bm{x},
\\
\bm{C}\bm{x}
&\leq
\bm{h}.
\end{align*}

Since for all $\bm{x}$ it holds that $\bm{x}^{-}\bm{A}\bm{x}\geq\theta$, the solution set for the system remains the same if we replace the first equality by the inequality $\bm{x}^{-}\bm{A}\bm{x}\leq\theta$. Moreover, it is easy to verify that for all regular $\bm{x}$ the new inequality is equivalent to the inequality $\theta^{-1}\bm{A}\bm{x}\leq\bm{x}$. Indeed, after left multiplication of the former inequality by $\theta^{-1}\bm{x}$, we have $\theta^{-1}\bm{A}\bm{x}\leq\theta^{-1}\bm{x}\bm{x}^{-}\bm{A}\bm{x}\leq\bm{x}$, which yields the latter inequality. At the same time, left multiplication of the inequality $\theta^{-1}\bm{A}\bm{x}\leq\bm{x}$ by $\theta\bm{x}^{-}$ leads to the inequality $\bm{x}^{-}\bm{A}\bm{x}\leq\theta\bm{x}^{-}\bm{x}=\theta$, and thus both inequalities are equivalent.

The above system now takes the form
\begin{align*}
\theta^{-1}\bm{A}\bm{x}
&\leq
\bm{x},
\\
\bm{B}\bm{x}\oplus\bm{g}
&\leq
\bm{x},
\\
\bm{C}\bm{x}
&\leq
\bm{h}.
\end{align*}

By combining the first two inequalities into one, we arrive at a system in the form of \eqref{I-AxbxCxd},
\begin{equation}
\begin{aligned}
(\theta^{-1}\bm{A}\oplus\bm{B})\bm{x}\oplus\bm{g}
&\leq
\bm{x},
\\
\bm{C}\bm{x}
&\leq
\bm{h}.
\end{aligned}
\label{S-theta1ABxgCxh}
\end{equation}

By Lemma~\ref{L-IAxbxCxd}, the system has regular solutions if and only if
\begin{equation}
\mathop\mathrm{Tr}(\theta^{-1}\bm{A}\oplus\bm{B})\oplus\bm{h}^{-}\bm{C}(\theta^{-1}\bm{A}\oplus\bm{B})^{\ast}\bm{g}
\leq
\mathbb{1}.
\label{I-Trtheta1ABhCtheta1ABg}
\end{equation}

With \eqref{E-TrAdCAb}, the left-hand side in inequality \eqref{I-Trtheta1ABhCtheta1ABg} can be written in another form
$$
\bm{h}^{-}\bm{C}\bm{g}
\oplus
\bigoplus_{m=1}^{n}\mathop\mathrm{tr}((\theta^{-1}\bm{A}\oplus\bm{B})^{m}(\bm{I}\oplus\bm{g}\bm{h}^{-}\bm{C}))
\leq
\mathbb{1}.
$$

To further rearrange the inequality, we write a binomial identity
$$
(\theta^{-1}\bm{A}\oplus\bm{B})^{m}
=
\bm{B}^{m}
\oplus
\bigoplus_{k=1}^{m}\mathop{\bigoplus\hspace{1.2em}}_{i_{0}+i_{1}+\cdots+i_{k}=m-k}\theta^{-k}\bm{B}^{i_{0}}(\bm{A}\bm{B}^{i_{1}}\cdots\bm{A}\bm{B}^{i_{k}}).
$$

Substitution of the identity together with some algebra result in
\begin{multline*}
\bigoplus_{m=1}^{n}\mathop\mathrm{tr}\bm{B}^{m}
\oplus
\bigoplus_{m=0}^{n}\bm{h}^{-}\bm{C}\bm{B}^{m}\bm{g}
\\
\oplus
\bigoplus_{m=1}^{n}\bigoplus_{k=1}^{m}\mathop{\bigoplus\hspace{1.2em}}_{i_{0}+i_{1}+\cdots+i_{k}=m-k}\theta^{-k}\mathop\mathrm{tr}(\bm{B}^{i_{0}}(\bm{A}\bm{B}^{i_{1}}\cdots\bm{A}\bm{B}^{i_{k}})(\bm{I}\oplus\bm{g}\bm{h}^{-}\bm{C}))
\leq
\mathbb{1}.
\end{multline*}

Consider the first two terms on the left. Note that $\mathop\mathrm{Tr}(\bm{B})\leq\mathbb{1}$, and thus $\bm{B}^{n}\leq\bm{B}^{\ast}$. Therefore, we have
$$
\bigoplus_{m=1}^{n}\mathop\mathrm{tr}\bm{B}^{m}
\oplus
\bigoplus_{m=0}^{n}\bm{h}^{-}\bm{C}\bm{B}^{m}\bm{g}
=
\mathop\mathrm{Tr}(\bm{B})
\oplus
\bm{h}^{-}\bm{C}\bm{B}^{\ast}\bm{g}.
$$

Since the inequality $\mathop\mathrm{Tr}(\bm{B})\oplus\bm{h}^{-}\bm{C}\bm{B}^{\ast}\bm{g}\leq\mathbb{1}$ is provided by the conditions of the theorem, these terms can be eliminated to write inequality \eqref{I-Trtheta1ABhCtheta1ABg} in the form 
$$
\bigoplus_{m=1}^{n}\bigoplus_{k=1}^{m}\mathop{\bigoplus\hspace{1.2em}}_{i_{0}+i_{1}+\cdots+i_{k}=m-k}\theta^{-k}\mathop\mathrm{tr}(\bm{B}^{i_{0}}(\bm{A}\bm{B}^{i_{1}}\cdots\bm{A}\bm{B}^{i_{k}})(\bm{I}\oplus\bm{g}\bm{h}^{-}\bm{C}))
\leq
\mathbb{1}.
$$

After rearranging terms, we get an inequality
$$
\bigoplus_{k=1}^{n}\mathop{\bigoplus\hspace{0.0em}}_{0\leq i_{0}+i_{1}+\cdots+i_{k}\leq n-k}\theta^{-k}\mathop\mathrm{tr}(\bm{B}^{i_{0}}(\bm{A}\bm{B}^{i_{1}}\cdots\bm{A}\bm{B}^{i_{k}})(\bm{I}\oplus\bm{g}\bm{h}^{-}\bm{C}))
\leq
\mathbb{1},
$$
which is equivalent to a system of inequalities
$$
\mathop{\bigoplus\hspace{0.0em}}_{0\leq i_{0}+i_{1}+\cdots+i_{k}\leq n-k}\theta^{-k}\mathop\mathrm{tr}(\bm{B}^{i_{0}}(\bm{A}\bm{B}^{i_{1}}\cdots\bm{A}\bm{B}^{i_{k}})(\bm{I}\oplus\bm{g}\bm{h}^{-}\bm{C}))
\leq
\mathbb{1},
\quad
k=1,\ldots,n.
$$

By solving each inequality in the system and then combining the solutions into one, we arrive at a lower bound for $\theta$, which is given by
$$
\theta
\geq
\bigoplus_{k=1}^{n}\mathop{\bigoplus\hspace{0.0em}}_{0\leq i_{0}+i_{1}+\cdots+i_{k}\leq n-k}\mathop\mathrm{tr}\nolimits^{1/k}(\bm{B}^{i_{0}}(\bm{A}\bm{B}^{i_{1}}\cdots\bm{A}\bm{B}^{i_{k}})(\bm{I}\oplus\bm{g}\bm{h}^{-}\bm{C})).
$$

Since $\theta$ is assumed to be the minimum of the objective function in the problem, the last inequality must be satisfied as an equality, which leads to \eqref{E-theta}.

Application of Lemma~\ref{L-IAxbxCxd} to system \eqref{S-theta1ABxgCxh} gives the solution vector $\bm{x}$ that is defined by \eqref{E-xthetaABu-guhCthetaAB}.
\end{proof}

\subsection{Special cases}

In this section we discuss problems that present noteworthy particular cases of the general problem examined above. Another special case is examined in the next section in the context of solving scheduling problems.

First, we assume $\bm{C}=\mathbb{0}$ and consider a problem given by
\begin{equation}
\begin{aligned}
&
\text{minimize}
&&
\bm{x}^{-}\bm{A}\bm{x},
\\
&
\text{subject to}
&&
\bm{B}\bm{x}\oplus\bm{g}
\leq
\bm{x}.
\end{aligned}
\label{P-xAxBxgx}
\end{equation}

A slight rearranging of the proof in Theorem~\ref{T-xAxBxgxCxh} leads to the following solution with a simplified expression for $\theta$ instead of that of \eqref{E-theta}.

\begin{corollary}\label{C-xAxBxgx}
Suppose that $\bm{A}$ is a matrix with spectral radius $\lambda>\mathbb{0}$, $\bm{B}$ is a matrix with $\mathop\mathrm{Tr}(\bm{B})\leq\mathbb{1}$, and $\bm{g}$ is a vector. Define a scalar
$$
\theta
=
\lambda
\oplus
\bigoplus_{k=1}^{n-1}\mathop{\bigoplus\hspace{1.2em}}_{1\leq i_{1}+\cdots+i_{k}\leq n-k}\mathop\mathrm{tr}\nolimits^{1/k}(\bm{A}\bm{B}^{i_{1}}\cdots\bm{A}\bm{B}^{i_{k}}).
$$

Then the minimum in \eqref{P-xAxBxgx} is equal to $\theta$ and attained if and only if
$$
\bm{x}
=
(\theta^{-1}\bm{A}\oplus\bm{B})^{\ast}\bm{u},
\qquad
\bm{u}
\geq
\bm{g}.
$$
\end{corollary}

Note that this result is coincides with that in \cite{Krivulin2013Amultidimensional}.

Finally, suppose that $\bm{B}=\mathbb{0}$ and $\bm{C}=\bm{I}$. Problem \eqref{P-xAxBxgxCxh} takes the form
\begin{equation}
\begin{aligned}
&
\text{minimize}
&&
\bm{x}^{-}\bm{A}\bm{x},
\\
&
\text{subject to}
&&
\bm{g}
\leq
\bm{x}
\leq
\bm{h}.
\end{aligned}
\label{P-xAxgxh}
\end{equation}

In this case, Theorem~\ref{T-xAxBxgxCxh} reduces to the next statement.

\begin{corollary}\label{C-xAxgxh}
Suppose that $\bm{A}$ is a matrix with spectral radius $\lambda>\mathbb{0}$, $\bm{g}$ is a vector, and $\bm{h}$ is a regular vector such that $\bm{h}^{-}\bm{g}\leq\mathbb{1}$. Define a scalar
$$
\theta
=
\lambda
\oplus
\bigoplus_{k=1}^{n}(\bm{h}^{-}\bm{A}^{k}\bm{g})^{1/k}.
$$

Then the minimum in \eqref{P-xAxgxh} is equal to $\theta$ and attained if and only if
$$
\bm{x}
=
(\theta^{-1}\bm{A})^{\ast}\bm{u},
\qquad
\bm{g}
\leq
\bm{u}
\leq
(\bm{h}^{-}(\theta^{-1}\bm{A})^{\ast})^{-}.
$$
\end{corollary}

\section{Applications to optimal scheduling}\label{S-AOS}

We start with a real-world problem taken from project scheduling and intended to both motivate and illustrate the results obtained. For further details and references on project scheduling, one can consult \cite{Demeulemeester2002Project,Tkindt2006Multicriteria}. 

We offer a vector representation of the problem in terms of tropical mathematics and then give a complete direct solution illustrated with a numerical example.

\subsection{Minimization of maximum flow time}

Consider a project with $n$ activities (jobs, tasks) constrained by precedence relations, including start-start, start-finish, early-start, and late-finish temporal constraints. For any two activities, the start-start constraints define the minimum allowed time interval between their initiations. The start-finish constraints place a lower bound on the time lag between the initiation of one activity and the completion of another. The activities are assumed to complete as early as possible within the constraints. For each activity, the early-start and late-finish constraints respectively specify the earliest possible time of initiation and the latest possible time of completion. 

Every activity in the project involves its associated flow (turnaround, processing) time defined as the time interval between its initiation and completion. The optimal scheduling problem is to find an initiation time for each activity to minimize the maximum flow time over all activities, subject to the above constraints.

For each activity $i=1,\ldots,n$, we denote by $x_{i}$ the initiation time to be scheduled. Let $g_{i}$ be a lower bound on the initiation time, and $b_{ij}$ be a minimum possible time lag between the initiation of activity $j=1,\ldots,n$ and the initiation of $i$.

The start-start constraints imply that, given the time lags $b_{ij}$, the initiation times are to satisfy the relations
$$
x_{j}+b_{ij}
\leq
x_{i},
\quad
j=1,\ldots,n.
$$ 

Note that if a time lag is not actually fixed, we set it to be equal to $-\infty$.
 
These relations taken together lead to one inequality of the form
$$
\max(x_{1}+b_{i1},\ldots,x_{n}+b_{in})
\leq
x_{i}.
$$

Since, according to the early-start constraints, activity $i$ cannot start earlier than at a predefined time $g_{i}$, we arrive at the inequalities
\begin{equation*}
\max(b_{i1}+x_{1},\ldots,b_{in}+x_{n},g_{i})
\leq
x_{i},
\quad
i=1,\ldots,n.
\end{equation*}

Furthermore, for each activity $i$, let $y_{i}$ be the completion time. We denote by $a_{ij}$ a given minimum possible time lag between the initiation of activity $j$ and the completion of $i$, and by $h_{i}$ a given upper bound on the completion time for $i$. As before, if a time lag $a_{ij}$ appears to be undefined, we put $a_{ij}=-\infty$.

The start-finish constraints require that the completion time $y_{i}$ be subject to the relations
$$
x_{j}+a_{ij}\leq y_{i},
\quad
j=1,\ldots,n,
$$ 
with at least one inequality among them holding as an equality.

By combining the inequalities and adding the upper bound for the completion time, we get the relations
\begin{equation*}
\max(a_{i1}+x_{1},\ldots,a_{in}+x_{n})
=
y_{i},
\quad
h_{i}
\geq
y_{i},
\quad
i=1,\ldots,n.
\end{equation*}

We now formulate a scheduling problem to minimize the maximum flow time over all activities. With an objective function that is readily given by
$$
\max(y_{1}-x_{1},\ldots,y_{n}-x_{n}),
$$
we arrive at a constrained optimization problem to find $x_{i}$ for all $i=1,\ldots,n$ to
\begin{equation}
\begin{aligned}
&
\text{minimize}
&&
\max(y_{1}-x_{1},\ldots,y_{n}-x_{n}),
\\
&
\text{subject to}
&&
\max(b_{i1}+x_{1},\ldots,b_{in}+x_{n},g_{i})
\leq
x_{i},
\\
&
&&
\max(a_{i1}+x_{1},\ldots,a_{in}+x_{n})
=
y_{i},
\quad
h_{i}
\geq
y_{i},
\quad
i=1,\ldots,n.
\end{aligned}
\label{P-yxbxxgxaxyhy}
\end{equation}

\subsection{Representation of scheduling problem}

Since the representation of the problem given by \eqref{P-yxbxxgxaxyhy} involves only usual operations $\max$, addition, and additive inversion, we can translate it into the language of the semifield $\mathbb{R}_{\max,+}$.

First, we replace the standard operations at \eqref{P-yxbxxgxaxyhy} by their tropical counterparts to write the problem in scalar terms as follows: 
\begin{equation*}
\begin{aligned}
&
\text{minimize}
&&
x_{1}^{-1}y_{1}\oplus\cdots\oplus x_{n}^{-1}y_{n},
\\
&
\text{subject to}
&&
b_{i1}x_{1}\oplus\cdots\oplus b_{in}x_{n}\oplus g_{i}
\leq
x_{i},
\\
&
&&
a_{i1}x_{1}\oplus\cdots\oplus a_{in}x_{n}
=
y_{i},
\quad
h_{i}
\geq
y_{i},
\quad
i=1,\ldots,n.
\end{aligned}
\end{equation*}

Furthermore, we introduce matrices and vectors
$$
\bm{A}
=
(a_{ij}),
\quad
\bm{B}
=
(b_{ij}),
\quad
\bm{g}
=
(g_{i}),
\quad
\bm{h}
=
(h_{i}),
\quad
\bm{x}
=
(x_{i}),
\quad
\bm{y}
=
(y_{i}).
$$

In matrix-vector notation, the problem is to find regular vectors $\bm{x}$ such that
\begin{equation}
\begin{aligned}
&
\text{minimize}
&&
\bm{x}^{-}\bm{y},
\\
&
\text{subject to}
&&
\bm{B}\bm{x}\oplus\bm{g}
\leq
\bm{x},
\\
&
&&
\bm{A}\bm{x}
=
\bm{y},
\quad
\bm{h}
\geq
\bm{y}.
\end{aligned}
\label{P-xyBxgxAxyhy}
\end{equation}

\subsection{Solution of scheduling problem}

A complete direct solution to the scheduling problem is given in terms of the semifield $\mathbb{R}_{\max,+}$ by the next result.

\begin{theorem}\label{T-xyBxgxAxyhy}
Let $\bm{x}$ and $\bm{y}$ be the general regular solution of problem \eqref{P-xyBxgxAxyhy}, which involves a column-regular matrix $\bm{A}$ with a nonzero spectral radius, and a regular vector $\bm{h}$. Define $\Delta=\mathop\mathrm{Tr}(\bm{B})\oplus\bm{h}^{-}\bm{A}\bm{B}^{\ast}\bm{g}$ and
\begin{equation}
\theta
=
\bigoplus_{k=1}^{n}\mathop{\bigoplus\hspace{0.0em}}_{0\leq i_{0}+i_{1}+\cdots+i_{k}\leq n-k}\mathop\mathrm{tr}\nolimits^{1/k}(\bm{B}^{i_{0}}(\bm{A}\bm{B}^{i_{1}}\cdots\bm{A}\bm{B}^{i_{k}})(\bm{I}\oplus\bm{g}\bm{h}^{-}\bm{A})).
\label{E-theta1}
\end{equation}

Then the following statements are valid:
\begin{enumerate}
\item If $\Delta\leq\mathbb{1}$, then $\theta$ is the minimum in \eqref{P-xyBxgxAxyhy}, attained at
\begin{equation}
\bm{x}
=
\bm{S}^{\ast}\bm{u},
\qquad
\bm{y}
=
\bm{A}\bm{S}^{\ast}\bm{u},
\label{E-xy}
\end{equation}
where $\bm{S}=\theta^{-1}\bm{A}\oplus\bm{B}$, and $\bm{u}$ is any regular vector such that
\begin{equation}
\bm{g}\leq\bm{u}\leq(\bm{h}^{-}\bm{A}\bm{S}^{\ast})^{-};
\label{E-u}
\end{equation}
\item If $\Delta>\mathbb{1}$, then there is no regular solution.
\end{enumerate}
\end{theorem}
\begin{proof}
To solve problem \eqref{P-xyBxgxAxyhy}, we first eliminate the unknown vector $\bm{y}$ by the substitution $\bm{y}=\bm{A}\bm{x}$ wherever it appears. By this means, we arrive at a problem with respect to the vector $\bm{x}$, which takes the form of \eqref{P-xAxBxgxCxh} with $\bm{C}=\bm{A}$. Application of Theorem~\ref{T-xAxBxgxCxh} to the last problem gives a solution in terms of $\bm{x}$. Back substitution of the solution into the equality $\bm{y}=\bm{A}\bm{x}$ completes the proof.
\end{proof}

\subsection{Numerical example}

To illustrate the above result, we take an example project of three activities under constraints given by
$$
\bm{A}
=
\left(
\begin{array}{ccc}
4 & 0 & \mathbb{0}
\\
2 & 3 & 1
\\
1 & 1 & 3
\end{array}
\right),
\quad
\bm{B}
=
\left(
\begin{array}{rrr}
\mathbb{0} & -2 & 1
\\
0 & \mathbb{0} & 2
\\
-1 & \mathbb{0} & \mathbb{0}
\end{array}
\right),
\quad
\bm{g}
=
\left(
\begin{array}{c}
0
\\
0
\\
0
\end{array}
\right),
\quad
\bm{h}
=
\left(
\begin{array}{c}
5
\\
5
\\
5
\end{array}
\right),
$$
where the notation $\mathbb{0}=-\infty$ is used to save writing.

We start with verification of the existence conditions for regular solutions in Theorem~\ref{T-xyBxgxAxyhy}. We take the matrix $\bm{B}$ and calculate
$$
\bm{B}^{2}
=
\left(
\begin{array}{rrr}
0 & \mathbb{0} & 0
\\
1 & -2 & 1
\\
\mathbb{0} & -3 & 0
\end{array}
\right),
\quad
\bm{B}^{3}
=
\left(
\begin{array}{rrr}
-1 & -2 & 1
\\
0 & -1 & 2
\\
-1 & \mathbb{0} & -1
\end{array}
\right),
\quad
\bm{B}^{\ast}
=
\left(
\begin{array}{rrr}
0 & -2 & 1
\\
1 & 0 & 2
\\
-1 & -3 & 0
\end{array}
\right).
$$

Furthermore, we get $\mathop\mathrm{Tr}(\bm{B})=0$ and obtain
$$
\bm{A}\bm{B}^{\ast}
=
\left(
\begin{array}{ccc}
4 & 2 & 5
\\
4 & 3 & 5
\\
2 & 1 & 3
\end{array}
\right),
\qquad
\bm{h}^{-}\bm{A}\bm{B}^{\ast}
=
\left(
\begin{array}{ccc}
-1 & -2 & 0
\end{array}
\right),
\qquad
\bm{h}^{-}\bm{A}\bm{B}^{\ast}\bm{g}
=
0.
$$

Since we have $\mathop\mathrm{Tr}(\bm{B})\oplus\bm{h}^{-}\bm{A}\bm{B}^{\ast}\bm{g}=0=\mathbb{1}$, we conclude that the problem under study has regular solutions.

To get the solutions, we need to evaluate $\theta$ which is given by \eqref{E-theta1}. Considering that $n=3$, we represent $\theta$ with three terms as follows
$$
\theta
=
\mathop\mathrm{tr}\nolimits(\bm{C}_{1}\bm{D})
\oplus
\mathop\mathrm{tr}\nolimits^{1/2}(\bm{C}_{2}\bm{D})
\oplus
\mathop\mathrm{tr}\nolimits^{1/3}(\bm{C}_{3}\bm{D}),
$$
where
\begin{gather*}
\bm{C}_{1}
=
\bm{A}\oplus\bm{B}\bm{A}\oplus\bm{A}\bm{B}\oplus\bm{B}^{2}\bm{A}\oplus\bm{B}\bm{A}\bm{B}\oplus\bm{A}\bm{B}^{2},
\\
\bm{C}_{2}
=
\bm{A}^{2}\oplus\bm{B}\bm{A}^{2}\oplus\bm{A}\bm{B}\bm{A}\oplus\bm{A}^{2}\bm{B},
\qquad
\bm{C}_{3}
=
\bm{A}^{3},
\qquad
\bm{D}
=
\bm{I}\oplus\bm{g}\bm{h}^{-}\bm{A}.
\end{gather*}

First, we calculate the matrices
$$
\bm{A}^{2}
=
\left(
\begin{array}{ccc}
8 & 4 & 1
\\
6 & 6 & 4
\\
5 & 4 & 6
\end{array}
\right),
\quad
\bm{A}^{3}
=
\left(
\begin{array}{ccc}
12 & 8 & 5
\\
10 & 9 & 7
\\
9 & 7 & 9
\end{array}
\right),
\quad
\bm{D}
=
\left(
\begin{array}{rrr}
0 & -2 & -2
\\
-1 & 0 & -2
\\
-1 & -2 & 0
\end{array}
\right).
$$

To obtain the first term in the representation of $\theta$, we successively find
\begin{gather*}
\bm{B}\bm{A}
=
\left(
\begin{array}{rrr}
2 & 2 & 4
\\
4 & 3 & 5
\\
3 & -1 & \mathbb{0}
\end{array}
\right),
\quad
\bm{A}\bm{B}
=
\left(
\begin{array}{rrr}
0 & 2 & 5
\\
3 & 0 & 5
\\
2 & -1 & 3
\end{array}
\right),
\quad
\bm{B}^{2}\bm{A}
=
\left(
\begin{array}{ccc}
4 & 1 & 3
\\
5 & 2 & 4
\\
1 & 1 & 3
\end{array}
\right),
\\
\bm{B}\bm{A}\bm{B}
=
\left(
\begin{array}{rrr}
3 & 0 & 4
\\
4 & 2 & 5
\\
-1 & 1 & 4
\end{array}
\right),
\qquad
\bm{A}\bm{B}^{2}
=
\left(
\begin{array}{rrr}
4 & -2 & 4
\\
4 & 1 & 4
\\
2 & 0 & 3
\end{array}
\right).
\end{gather*}

With the above matrices, we have
$$
\bm{C}_{1}
=
\left(
\begin{array}{ccc}
4 & 2 & 5
\\
5 & 3 & 5
\\
3 & 1 & 4
\end{array}
\right),
\quad
\bm{C}_{1}\bm{D}
=
\left(
\begin{array}{ccc}
4 & 3 & 5
\\
5 & 3 & 5
\\
3 & 2 & 4
\end{array}
\right),
\quad
\mathop\mathrm{tr}\nolimits(\bm{C}_{1}\bm{D})
=
4.
$$

Furthermore, we compute the matrices
$$
\bm{B}\bm{A}^{2}
=
\left(
\begin{array}{rrr}
6 & 5 & 7
\\
8 & 6 & 8
\\
7 & 3 & 0
\end{array}
\right),
\quad
\bm{A}\bm{B}\bm{A}
=
\left(
\begin{array}{rrr}
6 & 6 & 8
\\
7 & 6 & 8
\\
6 & 4 & 6
\end{array}
\right),
\quad
\bm{A}^{2}\bm{B}
=
\left(
\begin{array}{rrr}
4 & 6 & 9
\\
6 & 4 & 8
\\
5 & 3 & 6
\end{array}
\right),
$$
and then find
$$
\bm{C}_{2}
=
\left(
\begin{array}{rrr}
8 & 6 & 9
\\
8 & 6 & 8
\\
7 & 4 & 6
\end{array}
\right),
\quad
\bm{C}_{2}\bm{D}
=
\left(
\begin{array}{ccc}
8 & 7 & 9
\\
8 & 6 & 8
\\
7 & 5 & 6
\end{array}
\right),
\quad
\bm{C}_{3}\bm{D}
=
\left(
\begin{array}{ccc}
12 & 10 & 10
\\
10 & 9 & 8
\\
9 & 7 & 9
\end{array}
\right).
$$

After evaluating the second and third terms, we get
$$
\mathop\mathrm{tr}\nolimits(\bm{C}_{2}\bm{D})
=
8,
\qquad
\mathop\mathrm{tr}\nolimits(\bm{C}_{3}\bm{D})
=
12,
\qquad
\theta
=
4.
$$

We now derive the solution vectors $\bm{x}$ and $\bm{y}$ according to \eqref{E-xy} and \eqref{E-u}. First, we compute the matrices
$$
\bm{S}
=
\left(
\begin{array}{rrr}
0 & -2 & 1
\\
0 & -1 & 2
\\
-1 & -3 & -1
\end{array}
\right),
\quad
\bm{S}^{2}
=
\left(
\begin{array}{rrr}
0 & -2 & 1
\\
1 & -1 & 1
\\
-1 & -3 & 0
\end{array}
\right),
\quad
\bm{S}^{\ast}
=
\left(
\begin{array}{rrr}
0 & -2 & 1
\\
1 & 0 & 2
\\
-1 & -3 & 0
\end{array}
\right).
$$

We take the last matrix to get
$$
\bm{A}\bm{S}^{\ast}
=
\left(
\begin{array}{ccc}
4 & 2 & 5
\\
4 & 3 & 5
\\
2 & 1 & 3
\end{array}
\right),
\quad
\bm{h}^{-}\bm{A}\bm{S}^{\ast}
=
\left(
\begin{array}{ccc}
-1 & -2 & 0
\end{array}
\right),
\quad
(\bm{h}^{-}\bm{A}\bm{S}^{\ast})^{-}
=
\left(
\begin{array}{c}
1
\\
2
\\
0
\end{array}
\right).
$$

Denote by $\bm{u}_{1}$ and $\bm{u}_{2}$ the lower and upper bounds for the vector $\bm{u}$, which are defined by \eqref{E-u}, and write
$$
\bm{u}_{1}
=
\left(
\begin{array}{c}
0
\\
0
\\
0
\end{array}
\right),
\qquad
\bm{u}_{2}
=
\left(
\begin{array}{c}
1
\\
2
\\
0
\end{array}
\right).
$$

The bounds on the vector $\bm{u}$ produce corresponding bounds $\bm{x}_{1}$ and $\bm{x}_{2}$ on the vector $\bm{x}$. Evaluating the bounds on $\bm{x}$ gives
$$
\bm{x}_{1}
=
\bm{S}^{\ast}\bm{u}_{1}
=
\left(
\begin{array}{c}
1
\\
2
\\
0
\end{array}
\right),
\qquad
\bm{x}_{2}
=
\bm{S}^{\ast}\bm{u}_{2}
=
\left(
\begin{array}{c}
1
\\
2
\\
0
\end{array}
\right).
$$

Since these bounds actually define a single vector, we arrive at a unique solution to the problem
$$
\bm{x}
=
\left(
\begin{array}{c}
1
\\
2
\\
0
\end{array}
\right),
\qquad
\bm{y}
=
\bm{A}\bm{x}
=
\left(
\begin{array}{c}
5
\\
5
\\
3
\end{array}
\right).
$$

\subsection*{Acknowledgments} The author thanks the reviewer and an editor for valuable comments and suggestions, which have been incorporated into the final version.

\bibliographystyle{utphys}

\bibliography{A_constrained_tropical_optimization_problem_complete_solution_and_application_example}

\end{document}